\documentclass{amsart}

\usepackage{stmaryrd}
\usepackage{graphicx}
\usepackage{color}
\usepackage{float}
\usepackage[pdfpagelabels]{hyperref}

\newtheorem{theorem}{Theorem}[section]
\newtheorem{proposition}[theorem]{Proposition}
\newtheorem{corollary}[theorem]{Corollary}
\newtheorem{lemma}[theorem]{Lemma}

\theoremstyle{definition}
\newtheorem{definition}[theorem]{Definition}
\newtheorem{remark}[theorem]{Remark}
\newtheorem{conjecture}[theorem]{Conjecture}

\newtheorem*{xtheorem}{Theorem}
\newtheorem*{xdefinition}{Definition}

\newcommand{\A}{{\mathcal A}}
\newcommand{\B}{{\mathcal B}}
\newcommand{\C}{{\mathcal C}}
\newcommand{\D}{{\mathcal D}}
\newcommand{\E}{{\mathcal E}}
\newcommand{\LL}{{\mathcal L}}
\newcommand{\M}{{\mathcal M}}
\newcommand{\R}{{\mathcal R}}
\newcommand{\T}{{\mathcal T}}
\newcommand{\U}{{\mathcal U}}
\newcommand{\NN}{\mathbb{N}}
\newcommand{\RR}{\mathbb{R}}
\newcommand{\ZZ}{\mathbb{Z}}
\newcommand{\BMO}{\text{BMO}}
\newcommand{\BMON}{\text{BMO}_N}

\newcommand{\car}[1]{{\llbracket #1 \rrbracket}} 
\newcommand{\abs}[1]{{\left|#1\right|}} 
\newcommand{\norm}[1]{{\left\|#1\right\|}}

\newcommand{\normBMO}[2]{{\left\|#1\right\|^{#2}_{\text{BMO}}}}

\newcommand{\normBMON}[2]{{\left\|#1\right\|^{#2}_{\text{BMO}_N}}}

\DeclareMathOperator{\Log2}{log_2}
\DeclareMathOperator{\Level}{Level}
\DeclareMathOperator{\Pos}{Pos}
\DeclareMathOperator{\lin}{span}

\numberwithin{equation}{section}

\begin{document}
\title[Postorder rearrangement operators]{Postorder rearrangement operators}
\author{Johanna Penteker}
%\keywords{-}

\address{J. Penteker, Institute of Analysis, Johannes Kepler University Linz, Austria, 4040 Linz, Altenberger Strasse 69}
\email{johanna.penteker@jku.at}
\date{\today}

\maketitle
\setcounter{tocdepth}{1}
%\tableofcontents

\subsection*{Abstract}
We investigate the rearrangement of the Haar system induced by the postorder on the set of dyadic intervals in $[0,1]$ with length greater than or equal to $2^{-N}$.
By means of operator norms on $\BMON$ we prove that the postorder has maximal distance to the usual lexicographic order.

\section{Introduction}
Let $\D_N$ be the set of dyadic intervals in $[0,1]$ with length greater than or equal to $2^{-N}$. Let $\tau$ be any bijective map on $\D_N$ and $(h_I)_{I \in \D_N}$ the $L^{\infty}$-normalised Haar system.   
On the space $\BMON$ we consider rearrangements of the Haar system induced by the map $\tau$:
\[T_{\tau}:h_I \rightarrow h_{\tau(I)}.\]

In recent years boundedness criteria and extrapolation properties for rearrangement operators that rearrange the Haar system have been studied in detail. See, \cite{MR510261,MR648492,MR1063121,MR1466662,MR2183484,MR2559130,MR2927805,MR3054318}.

In the present work we complement the cited papers by investigating in detail one particular rearrangement and its extremal nature. 
We introduce the \textit{postorder}, $\preceq$, on the set of dyadic intervals $\D_N$.
\begin{xdefinition}
 Let $I,J \in \D_N$. We say $I \preceq J$ if either $I$ and $J$ are disjoint and $I$ is to the left of $J$, or $I$ is contained in $J$. 
\end{xdefinition}  
\noindent
This specific order defines a bijective map $\tau_N$ on the set $\D_N$, called the \textit{postorder rearrangement}, that maps the $n^{th}$ interval in postorder onto the $n^{th}$ interval in lexicographic order. Its inverse is denoted by $\sigma_N$.

We show that the postorder has maximal distance to the usual lexicographic order on $\D_N$. We quantify the distance by the product of operator norms
 \[\norm{T_{\tau_N}:\, \BMON \rightarrow \BMON}\norm{T_{\sigma_N}:\, \BMON \rightarrow \BMON}.\]
Particularly, we prove that within a factor of $\sqrt{2}$, on $\BMON$, both the operator $T_{\tau_N}$ and its inverse $T_{\sigma_N}$ reach maximal norm. 
We denote
 \[R^N(\BMON)=\sup\Big\{\norm{T_{\tau}\colon \BMON \rightarrow \BMON}:\,\tau\colon \D_N \rightarrow \D_N\,\, \text{bijective}\Big\}.\]
Our main result is 
 \begin{xtheorem}
  For $T=T_{\tau_N}$ and $T=T_{\sigma_N}$ we have 
 \begin{equation*}
  \frac{1}{\sqrt{2}}R^N(\BMON)\leq \norm{T}_{\BMON}\leq R^N(\BMON).  
 \end{equation*}
  \end{xtheorem}
This continuous the previous study of \cite{MR1389531}, who determine from a different perspective the extremal nature of the postorder and the induced rearrangement. P.F.X.~M\"uller and G.~Schechtman show that any block basis of the Haar system $(h_I)_{I \in \D_N}$ with respect to the postorder, $\preceq$, spans spaces that are well isomorphic to $\ell^p_k$, $1<p\neq2<\infty.$ On the other hand it is easy to find block bases of the Haar system with respect to the lexicographic order (the Rademacher functions) whose span is well isomorphic to $\ell^2_k$.

\medskip 
The postorder has its origin in computer sciences (see e.g.~\cite{berman,MR2245382}). In computer sciences, especially in the design and analysis of algorithms, dyadic trees are commonly used data structures, which enable efficient access to data. Tree traversal algorithms, which 
systematically walk through a tree and visit each node exactly once, enhance this efficient access. These algorithms define a specific order on the nodes of a tree. This makes it possible to talk about the node following or preceding a given one. The postorder tree traversal visits the left child, then the right child and then the node itself. Considering the dyadic tree structure of $\D_N$ this traversal induces exactly the postorder, $\preceq$, on $\D_N$.
The postorder tree traversal is for example used in the mergesort algorithm, invented by von Neumann in 1945.   
A more basic application is deallocating memory of all nodes of a tree, i.e.~deleting a tree. In calculator programs the postorder tree traversal is used to evaluate postfix notation.  
 
%\medskip
The Mallat algorithm for discrete wavelet transform (DWT) (see \cite{Mallat1, MR1219953}) determines the wavelet coefficients of a given discrete signal in a specific order which works its way up from the finest level to the coarsest. In case of the Haar transform this order is exactly our postorder, $\preceq$, cf.~figure \ref{fig:wave}. We discuss the discrete Haar wavelet transform (see e.g.~\cite{MR2400818}) now in detail. 
	\begin{figure}[h]
	 	\centering
	 		\includegraphics[width=0.9\textwidth,viewport=130 480 490 650]{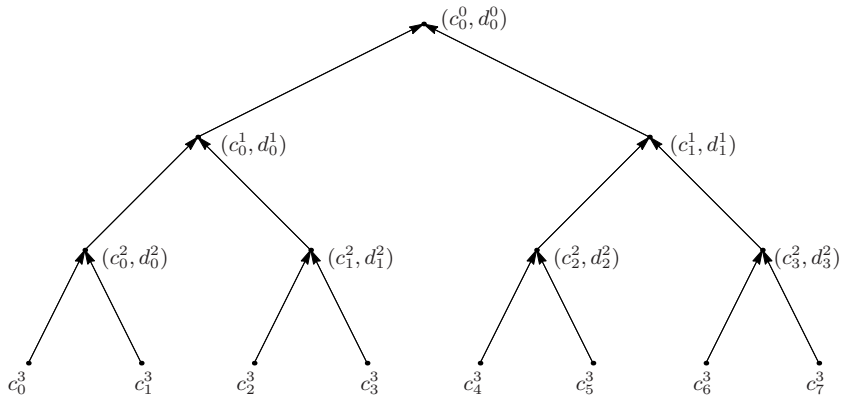}
	 	\caption{Calculation of the trend $(c^j_{k})$ and the fluctuation $(d^j_{k})$ for $0 \leq j \leq 3$.}
	 	\label{fig:wave}	
	\end{figure}
Let $N \in \NN_0$. Suppose a discrete signal on $[0,1]$ is given by the sequence $c^N=(c^N_1,\dots c^N_{2^{N-1}})$.
We process the signal by decomposing it into its trend (approximating coefficients) $c^{N-1}$ and its fluctuation (detail coefficients) $d^{N-1}$:
	\begin{align*}
	c^{N-1}_k=\frac{1}{\sqrt{2}}\left(c^N_{2k}+c^N_{2k+1}\right)\quad \text{and}\quad
	d^{N-1}_k=\frac{1}{\sqrt{2}}\left(c^N_{2k}-c^N_{2k+1}\right).
	\end{align*}
The trend and the fluctuation are two subsignals of $c^N$ with half of its length. 
The signal $c^{N-1}$ is again decomposed into its trend $c^{N-2}$ and its fluctuation $d^{N-2}$, which are again subsignals of $c^{N-1}$ with half of its length. 
Successively we compute from $c^{j}$ the trends $c^{j-1}$ and the fluctuations $d^{j-1}$. 
Finally, after $N$ steps, we have decomposed the signal $c^N$ into the coarsest information $c^0_0$ and the detail coefficients $(d^j)_{j=0}^{N-1}$, where $d^j=(d^{j}_{0}, \dots, d^{j}_{2^j-1})$.

The Haar system is the most basic orthonormal wavelet basis used in DWT and gives insight in more sophisticated wavelet transforms. 
%The significance of the DHWT is that the Haar wavelet basis leads to easy calculations of the wavelet decomposition and reconstruction. 

\section{Preliminaries}
Throughout this paper we will denote by $\NN$ the set of positive integers and by $\NN_0=\NN\,\cup\,\{0\}$ the set of non-negative integers.
\[\text{Unless stated otherwise: } \ell,\, k, \, N \in \NN_0 \, \text{ such that } 0\leq \ell\leq N \text{ and } 0\leq k \leq 2^{\ell-1}.\]

\subsection{Floor and ceiling function}
The \textit{floor function} $\lfloor \cdot \rfloor\colon \RR \rightarrow \ZZ$ and the \textit{ceiling function} $\lceil \cdot \rceil\colon \RR \rightarrow \ZZ$ are defined as follows: 
\begin{align*}
\lfloor x \rfloor &= \max{\{z \in \ZZ: z\leq x\}},\quad 
\lceil x \rceil=\min{\{z \in \ZZ: z\geq x\}}.
\end{align*}

\subsection{Dyadic intervals and trees}
\subsubsection*{Dyadic intervals} 
 An interval $I \subseteq [0,1]$ is called a \textit{dyadic interval}, if there exist non-negative integers $\ell$ and $k$ with $0 \leq k \leq 2^{\ell}-1$ such that
	\[I=I_{{\ell},k}=\left[\frac{k}{2^{\ell}},\frac{k+1}{2^{\ell}}\right[.
\] 
The length of a dyadic interval $I_{\ell,k}$ is given by $\abs{I_{\ell,k}}=2^{-\ell}$.
In the following we consider for fixed $N \in \NN_0$ the set of dyadic intervals with length greater than or equal to $2^{-N}$ given by 
\begin{equation}
\label{eq:dn}
\D_N=\{I_{{\ell},k}: \, 0 \leq {\ell} \leq N,\,0 \leq  k \leq 2^{\ell}-1\}.
\end{equation}

\subsubsection*{Carleson constant} 
Let $\mathcal{C} \subseteq \D_N$. We define the Carleson constant of $\C$ as follows 
\begin{equation}
\car{\mathcal{C}}=\sup_{I\in \mathcal{C}}\frac{1}{\abs{I}}{\sum_{J \subseteq I, J \in\mathcal{C}}{\abs{J}}}.
\label{eq:carleson}
\end{equation} If $\mathcal{C}$ is non-empty, then $\car{\mathcal{C}}\geq 1$, otherwise $\car{ \mathcal{C} } = 0$.  

\subsubsection*{Dyadic trees}
\label{sec:binarytree}
 See \cite{berman, MR2245382}. 
A \textit{dyadic tree} $\mathcal{T}$ consists of a set of nodes that is either empty or has the following properties:
\begin{enumerate}
	\item One of the nodes, say $R$, is designated the root node.
	\item The remaining nodes (if any) are partitioned into two disjoint subsets, called the left subtree and the right subtree, respectively, each of which is a dyadic tree. 
\end{enumerate}
The definition yields that every node of a tree is the root of some
subtree contained in the tree $\T$. 
The root of the left resp.~the right subtree described in property (2) is called the \textit{left child} resp.~\textit{the right child} of the root $R$. Conversely, the root $R$ is called the \textit{parent} of the left (resp.~right) child. We use the terminology of family trees: parent, children, descendant, etc. 
The nodes of a dyadic tree ${\mathcal{T}}$ can be partitioned into disjoint sets, called \textit{levels}, depending on the length $\ell$ of the unique path from a node to the root $R$. The root $R$ is at level $0$. 
The \textit{lowermost level} of ${\mathcal{T}}$ is the set of nodes, whose unique path from the node to the root $\R$ has maximal length within the tree $\mathcal{T}$. 
The \textit{depth} of ${\mathcal{T}}$ is the number of levels in $\mathcal{T}$ that do not contain the root $R$. 
A dyadic tree $\mathcal{T}$ is \textit{complete}, if every node in $\mathcal{T}$ has exactly two children, except the nodes in the lowermost level, which have exactly zero children, cf.~figure \ref{fig:fig11}.
In the following we consider complete dyadic trees of depth $N$, $N\in \NN_0$. The number of nodes in each level $\ell$, $0 \leq \ell \leq N$, is given by $2^{\ell}$ and the total number of nodes in a complete dyadic tree of depth $N$ is given by $2^{N+1}-1$.

\subsubsection*{The complete dyadic tree $\D_N$}
 The set
 $\D_N$, given by equation (\ref{eq:dn}), has a natural dyadic tree structure, cf.~figure \ref{fig:fig11}.
\begin{figure}[h!]
	\centering
		\includegraphics[width=0.85\textwidth,viewport=125 490 480 680]{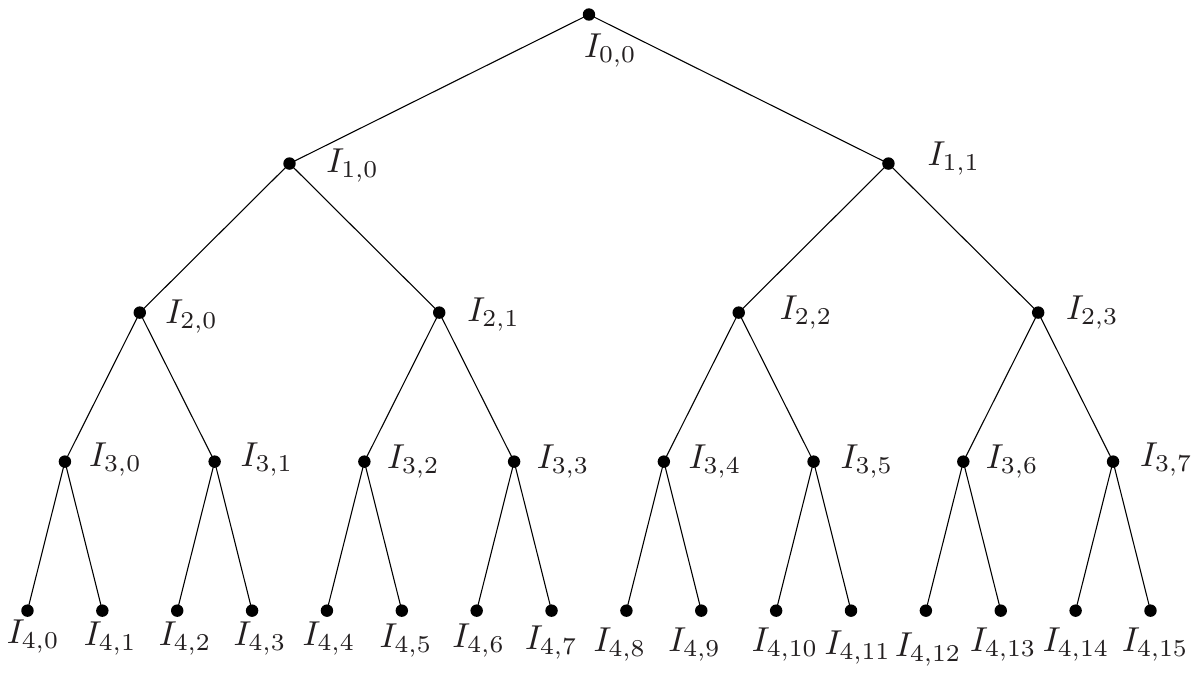}
	\caption{The dyadic tree structure of $\D_4$.}
	\label{fig:fig11}	
	\end{figure}
The root of the complete dyadic tree $\D_N$ is the dyadic interval $I_{0,0}$. The depth of $\D_N$ is equal to $N$. For an interval $I_{{\ell},k} \in \D_N$ the index ${\ell}$ denotes its level within the tree and $k$ its position within the level. The left resp.~the right child of an interval $I_{\ell,k} \in \D_N$ is given by
\begin{equation}
I_{\ell+1,2k}=\left[\frac{2k}{2^{\ell+1}},\frac{2k+1}{2^{\ell+1}}\right[\quad \text{resp.} \quad I_{\ell+1,2k+1}=\left[\frac{2k+1}{2^{\ell+1}},\frac{2\,(k+1)}{2^{\ell+1}}\right[.
\end{equation}

\subsubsection*{Dyadic subtrees}
\label{sec:defsubtrees}
Let $I_{\ell,k} \in \D_N$. We denote by $\mathcal{T}^N_{\ell,k}$ the complete dyadic subtree of $\D_N$ with root $I_{\ell,k}$ and depth $N-\ell$.
Note that $\mathcal{T}^N_{\ell,k}=\{I \in \D_N:\, I \subseteq I_{\ell,k}\}$.
Therefore, we get from (\ref{eq:carleson}) the Carleson constant
\begin{equation}
\label{eq:subtree}
\begin{split}
\car{ \mathcal{T}^N_{\ell,k}}&=\frac{1}{\abs{I_{\ell,k}}}\sum_{I \in \mathcal{T}^N_{\ell,k}}\abs{I}
%&=\frac{1}{\abs{I}}\abs{I}(N-\Log2{\frac{1}{\abs{I}}}+1)\\
=N-\ell+1. 
\end{split}
\end{equation}

\subsection{The order on $\D_N$}
\label{sec:postorder}
See \cite{MR1389531}, \cite{berman} and \cite{MR2245382}. 
The \textit{postorder} $\preceq$ on $\D_N$ is defined as follows. 
\begin{definition}
 Let $I,J \in \D_N$. We say $I \preceq J$ if either $I$ and $J$ are disjoint and $I$ is to the left of $J$, or $I$ is contained in $J$.
\end{definition}
In terms of the dyadic tree structure of $\D_N$ the postorder is defined as follows:
children are always smaller than their parent, the left child is always smaller than the right child and smaller than the descendants of the right child, cf. figure \ref{fig:fig1}.

The natural order on the set $\D_N$ is the \textit{lexicographic order}, $\leq_{l}$, on the set $\{({\ell},k)\}$. 
The postorder on $\D_N$, in contrast to the lexicographic order depends on the depth $N$.
The postorder works its way up from the leftmost node in the lowermost level to the root of the dyadic tree $\D_N$. Therefore, it is clear from the definition that the root of the dyadic tree $\D_N$ has postorder ordinal number $2^{N+1}-1$, which is the total number of nodes contained in the tree $\D_N$. 

Observe that $I_{1,0}$ is the left child and $I_{1,1}$ is the right child of the root $I_{0,0}$. 
Hence, the complete dyadic subtree  $\mathcal{T}^N_{1,0}$ resp.~$\mathcal{T}^N_{1,1}$ of $\D_N$ is the left resp.~right subtree of the root $I_{0,0}$. 
The definition of the postorder yields that the left subtree contains the ordinal numbers $1,\dots, 2^{N}-1$ and the right subtree the ordinal numbers $2^{N},\dots,2^{N+1}-2$. 

\subsection{The order intervals}
Let $J_1,J_2 \in \D_N$. An order interval with respect to the postorder, $\preceq$, is given by 
\begin{equation}
\label{eq:postint}
\B^N(J_1,J_2)=\{I \in \D_N: J_1 \preceq I \preceq J_2\},
\end{equation}
and with respect to the lexicographic order, $\leq_l$, by
\begin{equation}
\label{eq:lexint}
\E(J_1,J_2)=\{I \in \D_N: J_1 \leq_l I \leq_l J_2\}.
\end{equation}

The following definition and proposition is taken from \cite{MR1389531} and describes order intervals with respect to the postorder, $\preceq$.  
\begin{definition}
\label{def:cone}
Let $I,J \in \D_N$ with $I \subseteq J$. 
\begin{enumerate}
\item The \textit{cone} $\C=\C(I,J)$ of dyadic intervals between $I$ and $J$ is the unique collection of dyadic intervals $\C=\{C_1,\dots,C_n\}$, where $n=\Log2{\frac{\abs{J}}{\abs{I}}}+1$, satisfying
$C_1=I$, $C_n=J$, $\abs{C_i}=\frac{1}{2}\abs{C_{i+1}}$ and $C_i \subset C_{i+1}$ for $1\leq i \leq n-1$.

\item The \textit{right fill-up} of the cone $\C$ is the collection of dyadic intervals $\R=\R(I,J)=\bigcup_{i=1}^{n-1}\U_{i+1}$,
where $ \U_{i+1}=\{U \in \D_N: U \subseteq C_{i+1}\setminus C_{i}\}$,
if $C_i$ is the left half of $C_{i+1}$ and 
 $\U_{i+1}=\emptyset$, if $C_i$ is the right half of $C_{i+1}$.
%\[ \U_{i+1}=\{U \in \D_N: U \subseteq C_{i+1}\setminus C_{i}\}, \]
%if $C_i$ is the left half of $C_{i+1}$. 
\end{enumerate}
\end{definition}

\begin{figure}[h!]
	\centering
		\includegraphics[width=0.8\textwidth,viewport=118 496 491 692]{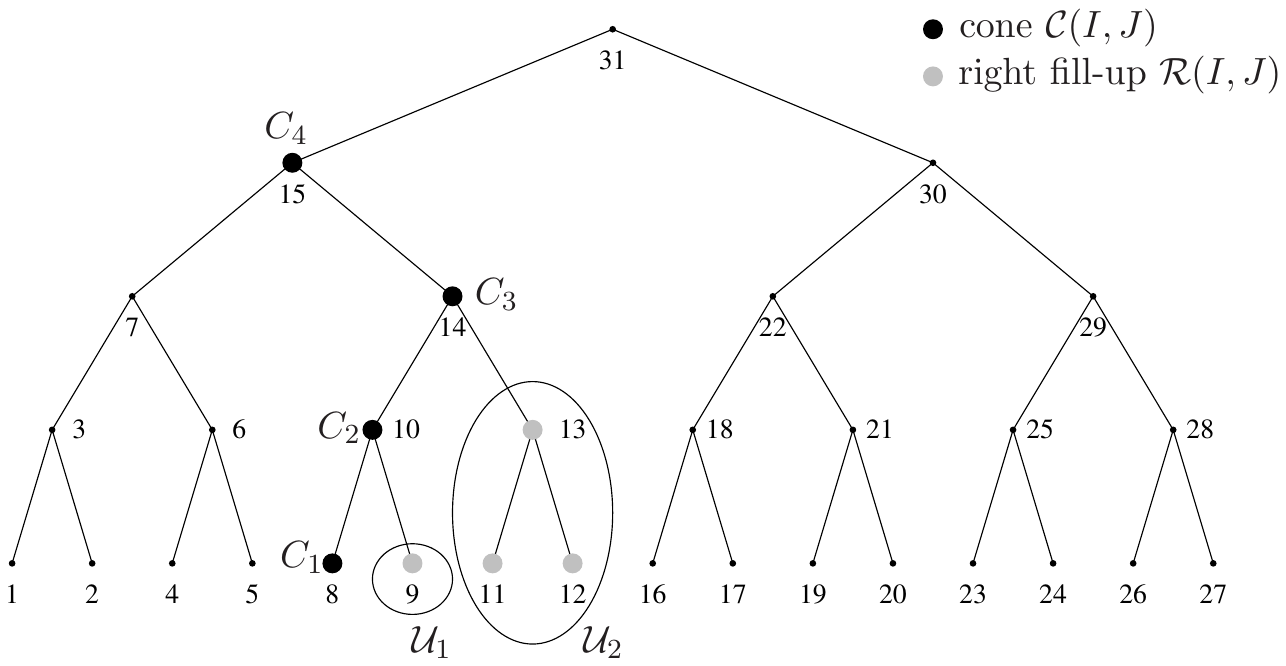}
	\caption{Cone and right fill-up given by the dyadic intervals $I=I_{4,4}$ and $J=I_{1,0}$ in $\D_4$.}
	\label{fig:figcr}
\end{figure}

\begin{proposition}
\label{prop:maxint}
Let $J_1,J_2 \in \D_N$ and $J_1 \preceq J_2$. 
For the postorder order interval $B^N(J_1,J_2)$
there exists a unique collection $\LL=\{L_1,\dots,L_m\}$ of pairwise disjoint dyadic intervals satisfying
\begin{enumerate}
\item $\abs{L_i}< \abs{L_{i-1}}$, if $2\leq i \leq m-1$;
\item $\abs{L_{m}}\leq \abs{L_{m-1}}$, if $m \geq 2$;
\item $L_{i+1}$ lies to the right of $L_i$ and the closures $\overline{L_i}$ and $\overline{L_{i+1}}$ intersect in exactly one point, the left endpoint of $\overline{L_{i+1}}$;
\item $J_1\subseteq L_1$, $J_2=L_m$ and 
\begin{equation}
\label{eq:propord}
\B^N(J_1,J_2)=\C(J_1,L_1)\cup\R(J_1,L_1)\cup_{i=2}^{m}\M_i,
\end{equation}
where $\M_i=\{I \in \D_N:I \subseteq L_i\}$. 
\end{enumerate}
\end{proposition}
\begin{remark}
Note that the intervals $(L_i)_{i=1}^m$ are the maximal (with respect to inclusion) dyadic intervals in the postorder order interval $B^N(J_1,J_2)$.
\end{remark}
\begin{figure}[h!]
\centering
	\includegraphics[width=0.85\textwidth,viewport=116 467 523 688]{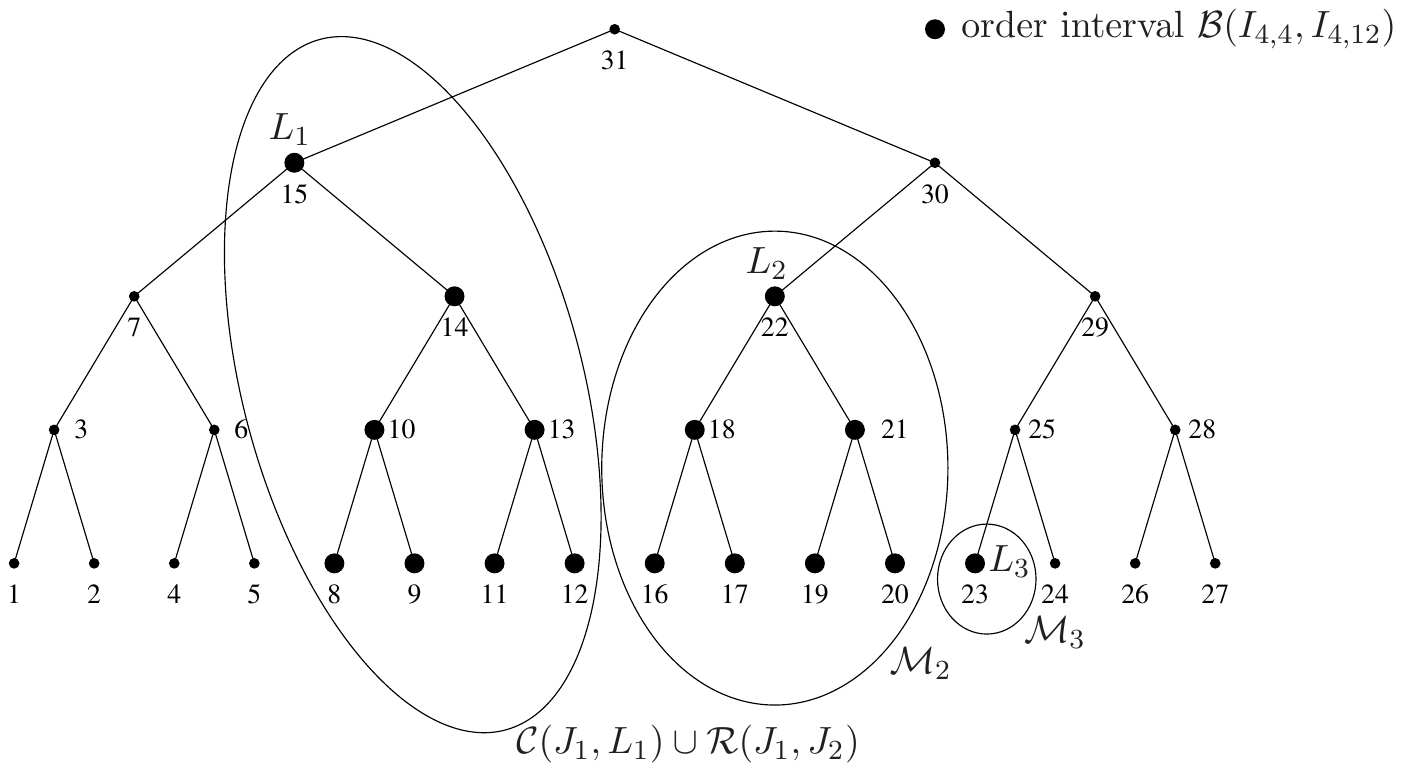}
		\caption{The postorder order interval $\B^4(I_{4,4},I_{4,12})$ in $\D_4$.}
		\label{fig:figorderint}
\end{figure}

\subsection{The spaces}
\subsubsection*{Haar system and Haar support} 
We define the $L^{\infty}$-normalised \textit{Haar system}
$(h_I)_{I\in \D_N}$ as follows:
	\[h_I=\begin{cases}
  1  &\text{ on the left half of $I$},\\
  -1  &\text{ on the right half of $I$},\\
  0  &\text{ otherwise.}
\end{cases}
\]

Let $(x_I)_{I \in \D_N}$ be a real sequence and let  
$f=\sum_{I \in \D_N}x_Ih_I$.
The \textit{Haar support} of $f$ is the set $\{I \in \D_N: x_I\neq 0\}$. The Haar support of $f$ is contained in a non-empty collection of dyadic intervals $\C \subseteq \D_N$ if and only if $f=\sum_{I \in \C}x_Ih_I$. 
We denote by 
$\M(\C)$ the space of all functions $f$ that have Haar support in a non-empty collection $\C \subseteq \D_N$.

\subsubsection*{Dyadic $\BMON$ and the dyadic Hardy spaces $H^p_N$.} 
We define here the known spaces $\BMON$ and $H^p_N$, for fixed $N \in \NN_0$, (see e.g.~\cite{MR955660}). Let $(x_I)_{I \in \D_N}$ be a real sequence and $f=\sum_{I \in \D_N}x_Ih_I$. We define
\begin{equation*}
\label{eq:bmonorm}
\normBMO{f}{} = \sup_{I\in \D_N}\bigg(\frac{1}{\abs{I}}\sum_{J\subseteq I}{\abs{x_J}^2\abs{J}} \bigg)^{\frac{1}{2}}
\end{equation*}
and 
\begin{equation*}
\label{eq:hpnorm}
	\norm{f}_{H^p}=\norm{S(f)}_{L^p([0,1])}, \quad \text{for }\,\,  0<p<\infty,
\end{equation*}
where $S(f)$ is the square function of $f$ defined by
\begin{equation*}
\label{eq:squaref}
	S(f)(t)=\bigg(\sum_{I \in \D_N}{\abs{x_I}^2 1_I(t)}\bigg)^{\frac{1}{2}}, \quad t \in [0,1].
\end{equation*}
Then we define the spaces $\BMON$ and $H^p_N$, $0<p<\infty$, as follows
\begin{equation*}
\BMON=\Big(\lin{\{h_I:\, I \in \D_N\}},\, \normBMO{\cdot}{}\Big)
\end{equation*}and
\begin{equation*}
H^p_N=\Big(\lin{\{h_I:\, I \in \D_N\}},\, \norm{\cdot}_{H^p}\Big).
\end{equation*}

Note that $\BMO_N$ and $H^p_N$ are finite dimensional subspaces of the dyadic $\BMO$ and $H^p$ spaces, defined in \cite{MR2157745}.

Paley's theorem (\cite{MR1576148}, see also \cite{MR2157745}) asserts that for all $1<p<\infty$ there exists a constant $A_p$ such that for all $f \in L^p([0,1])$ given by $f=\sum_{I\in \D}{x_Ih_I}$ 
\begin{equation*}
\label{eq:Paley}
\frac{1}{A_p}\norm{f}_{L^p}\leq\norm{S(f)}_{L^p}\leq A_p\norm{f}_{L^p}.
\end{equation*}
This theorem identifies $H^{q}$ as the dual space of $H^p$, where $\frac{1}{p}+\frac{1}{q}=1$ and $1<p<\infty$. 

Fefferman's inequality (\cite{MR0447953}, see also \cite{MR2157745})
\begin{equation}
\label{eq:feff}
\abs{\int {fh}\,}\leq 2\sqrt{2}\,\norm{f}_{H^1}\normBMO{h}{},
\end{equation} 
and a theorem to the effect that every continuous linear functional $L:H^1 \rightarrow \RR$ is necessarily of the form $L(f)=\int{f\varphi\,}$ with $\normBMO{\varphi}{}\leq \norm{L}$ identify BMO as dual space of $H^1$, (see \cite{MR0447953},\cite{MR0448538},\cite{MR2157745}). 

\subsection{The operators}
\label{sec:rearop}
\subsubsection*{Rearrangements of the Haar system}
Let $\tau$ be a bijective map defined on the set $\D_N$. 
On $\BMON$ we study rearrangements of the $L^{\infty}$-normalised Haar system $(h_I)_{I \in \D_N}$ given by the rearrangement operator
\begin{equation*}
T_{\tau}\colon h_I \mapsto h_{\tau(I)},
\label{eq:rbmo}
\end{equation*}
and on $H^p_N$, $0<p<\infty$, rearrangements of the $L^p$-normalised Haar system given by the rearrangement operator 
	\[T_{\tau,p}:\frac{h_I}{\abs{I}^{\frac{1}{p}}} \mapsto \frac{h_{\tau(I)}}{\abs{\tau(I)}^{\frac{1}{p}}}.
\]

A standard argument (given below) yields the following norm estimates for rearrangement operators on BMO
\begin{equation}
\label{eq:rear}
\sup_{\substack{\C \subseteq \D_N\\ \text{non-empty}}}{\frac{\car{\tau(\C)}^{\frac{1}{2}}}{\car{\C}^{\frac{1}{2}}}}\leq \normBMON{T_{\tau}}{}\leq \left(N+1\right)^{\frac{1}{2}}.
\end{equation}
Note that the lower bound in (\ref{eq:rear}) is always greater than or equal to one. 

Let $x=\sum_{I \in \D_N}x_Ih_I$. Then
\begin{align*}
\normBMO{T_{\tau}x}{2}&=\sup_{I \in \D_N}{\frac{1}{\abs{I}}\sum_{J\subseteq I}\abs{x_{\tau^{-1}(J)}}^2\abs{J}}\leq \sup_{I \in \D_N}{\abs{x_I}}^2\,\car{ \D_N}\leq \normBMO{x}{2}\,\car{ \D_N}. 
\end{align*}
Definition (\ref{eq:carleson}) yields $\car{ \D_N}=N+1$. 
Let $\mathcal{C}\subseteq \D_N$ be any non-empty collection of dyadic intervals.
Let $x=\sum_{I \in \mathcal{C}}{h_I}$. Then
\begin{equation*}
\normBMO{x}{}=\car{\mathcal{C}}^{\frac{1}{2}} \hspace{3pt} \text{ and } \hspace{3pt} \normBMO{T_{\tau}x}{}=\car{\tau(\mathcal{C})}^{\frac{1}{2}}.
\end{equation*}

Let $x=\sum_{I \in \C}{x_Ih_I}$ for some non-empty collection of dyadic intervals $\C \subseteq \D_N$. The above argument provides the following rough upper bound
\begin{equation}
\label{eq:rear1}
\normBMO{T_{\tau}x}{}\leq \normBMO{x}{}\,\car{ \tau(\C)}^{\frac{1}{2}}.
\end{equation}

The adjoint operator of a rearrangement operator is again a rearrangement operator induced by the inverse rearrangement.
By the duality of $H^1$ and BMO we have that the operator $T_{\tau}$ on $\BMON$ is the adjoint operator of $T_{\tau^{-1},1}$ on $H^1_N$ with
\begin{equation}
\label{eq:duality}
\frac{1}{C_F}\normBMON{T_{\tau}}{}\leq \norm{T_{\tau^{-1},1}}_{H^1_N}\leq C_F\normBMON{T_{\tau}}{},
\end{equation}where $C_F=2\sqrt{2}$ is the constant appearing in Fefferman's inequality (\ref{eq:feff}).

\subsubsection*{Interpolation and extrapolation of rearrangement operators}
See \cite{MR2183484, MR2157745}. 
The following interpolation resp.~extrapolation theorem provides a tool that enables one to deduce norm estimates for the rearrangement operators $T_{\tau,p}$ on $H^p_N$ for every $0<p<2$ from norm estimates of some rearrangement operator $T_{\tau,p_0}$ on $H^{p_0}_N$, $0<p_0<2$. The left-hand side inequality corresponds to an extrapolation based on Pisier's extrapolation norm (see \cite{MR2183484}). The right-hand side inequality is obtained by a standard interpolation argument. Note that $\norm{T_{\tau,2}}_{H^2_N}=1$.
\begin{theorem}
\label{th:interpolation0}
For all $0<s<r<2$ there exists a constant $c_{r,s}>0$ such that 
\begin{equation*}
\frac{1}{c_{r,s}}\norm{T_{\tau,s}}_{H^{s}_N}^{\frac{s}{2-s}}\leq \norm{T_{\tau,r}}_{H^{r}_N}^{\frac{r}{2-r}}\leq c_{r,s}\norm{T_{\tau,s}}_{H^{s}_N}^{\frac{s}{2-s}}.
\end{equation*}
\end{theorem}

The duality of $H^p$ and $H^{q}$, $1<q<2$, $\frac 1p +\frac 1{q}=1$, gives the following corollary to Theorem \ref{th:interpolation0}. Recall that the adjoint rearrangement operator on $H^{p}_N$ coincides with the inverse rearrangement operator on $H^q_N$.  
\begin{corollary}
\label{co:interpolation1}
For all $2<p<\infty$ there exists a constant $c_p$ such that 
\[\frac{1}{c_p}\norm{T_{\tau,p}}_{H^{p}_N}\leq \norm{T_{\tau^{-1},1}}_{H^1_N}^{1-\frac{2}{p}}\leq c_p\norm{T_{\tau,p}}_{H^{p}_N}.\]
\end{corollary}

\begin{remark}
\label{re:universalbound}
Observe that by the above theorem and corollary rearrangement operators $T_{\tau,p}$ on $H^p_N$, $0<p<\infty$, induced by any bijective map $\tau$ acting on $\D_{N}$, have the norm estimate 
\[\norm{T_{\tau,p}}_{H^p_N}\leq c_{p}\,{(N+1)^{\abs{\frac{1}{p}-\frac{1}{2}}}}.\]
\end{remark}

\section{The main theorem}
\label{sec:proof1}
Let $\tau_N$ be the bijective map on the dyadic intervals that associates to the $n^{th}$ interval in postorder the $n^{th}$ interval in lexicographic order, cf.~figure \ref{fig:fig1}. This rearrangement is called \textit{postorder rearrangement}. Its inverse, which associates to the $n^{th}$ interval in lexicographic order the $n^{th}$ interval in postorder, is denoted by $\sigma_N$.
\begin{figure}[H]
 	\centering
 			\includegraphics[width=0.85\textwidth,viewport=118 285 475 700]{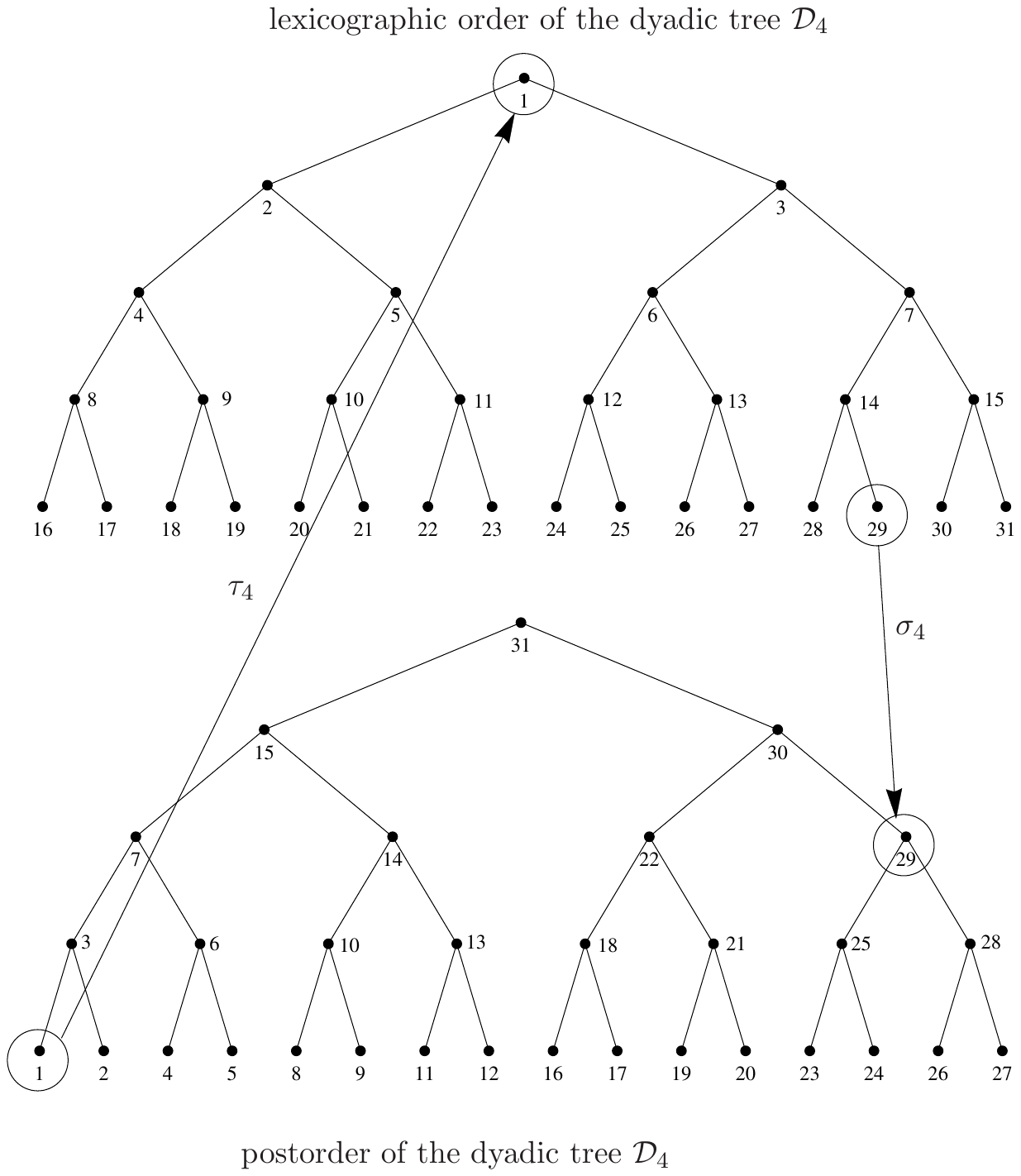}
 	\caption{Lexicographic order and postorder of the dyadic tree $\D_4$. Postorder rearrangement $\tau_4$ on $\D_4$ and its inverse $\sigma_4$.}
 	\label{fig:fig1}	
\end{figure}

The rearrangements $\tau_N$ and $\sigma_N$ induce rearrangement operators on $\BMON$ and on the $H^p_N$-spaces. 
On $\BMON$ we consider the rearrangement operators
\[T_{\tau_N}\colon h_I \mapsto h_{\tau_N(I)} \quad \text{and} \quad T_{\sigma_N}\colon h_I \mapsto h_{\sigma_N(I)}\]
and obtain the following norm estimates for these rearrangement operators applied to functions with Haar support in the sets $\T^N_{\ell,0}=\{I \in \D_N:\, I \subseteq I_{\ell,0}\}$ and $\D_{N-\ell}$.
Recall that $\M(\T^N_{\ell,0})=\lin{\{h_I: I \in \T^N_{\ell,0}\}}$ and $\M(\D_{N-\ell})=\lin{\{h_I: I \in \D_{N-\ell}\}}.$ 
\begin{theorem}
\label{th:operatornorm1}
Let $N \in \NN_0$ and $0\leq \ell \leq N$. Let $T=T_{\tau_N}\big\vert_{\M(\T_{\ell,0}^N)}$ or $T=T_{\sigma_N}\big\vert_{\M(\D_{N-\ell})}$.
Then
\begin{equation}
\label{eq:BMO1}
\frac{1}{\sqrt{2}}(N-\ell+1)^{\frac{1}{2}}\leq\normBMON{T}{}\leq (N-\ell+1)^{\frac{1}{2}}.
\end{equation}
\end{theorem} 
This theorem in combination with the general upper bound in (\ref{eq:rear}) reveals the extremal nature of the rearrangements $\tau_N$ and $\sigma_N$ in the sense that for $T=T_{\tau_N}$ resp. $T=T_{\sigma_N}$ we have
\begin{equation*}
\frac{1}{\sqrt{2}}R^N(\BMON)\leq  \normBMON{T}{} \leq R^N(\BMON),
\end{equation*}
where $R^N(\BMON)=\sup\Big\{\norm{T_{\tau}\colon \BMON \rightarrow \BMON}:\,\tau\colon \D_N \rightarrow \D_N\,\, \text{bijective}\Big\}.$

\noindent
Obviously, the lower bound in (\ref{eq:BMO1}) is the important one for this result and the statement of Theorem \ref{th:operatornorm1}.  The upper bound in (\ref{eq:BMO1}) is the trivial one that originates from the depth (in the sense of dyadic trees) of the sets $\D_{N-\ell}$ resp.~$\T^N_{\ell,0}$.

\medskip
Theorem \ref{th:lexorder} provides a tool that enables one to gain insight into the rearrangement operators $T_{\sigma_N}$ applied to spaces of functions with Haar support in a lexicographic order interval. 
 In Theorem \ref{th:operatornorm1} we have already seen that on the lexicographic order interval $\D_{N-\ell}$, for some small $\ell$, the operator has very large norm. Theorem \ref{th:lexorder} provides the possibility to determine canonical collections of dyadic intervals on which the rearrangement operator has small norm. The significance of the upper bound in Theorem \ref{th:lexorder} is given by the fact that $\Log2{\frac{1}{\abs{L_1}}}$ is able to compensate the term $N$. 
 In order to obtain the upper bound in Theorem \ref{th:lexorder} we use a geometric representation of order intervals with respect to the postorder, $\preceq$. Hence, one can read off the upper bound from the tree representation of $\D_N$, cf.~figure \ref{fig:figorderint}.  
 
\begin{theorem}
\label{th:lexorder}
Let $N \in \NN_0$. Let $\E=\E(E_1,E_2)$ be the lexicographic order interval given by the dyadic intervals $E_1,E_2 \in \D_N$ with $E_1 \leq_l E_2$. Then
\begin{equation}
\label{eq:Tsigmalex}
\normBMON{T_{\sigma_N}\big\vert_{\M(\E)}}{2}\leq N-\Log2{\frac{1}{\abs{L_1}}}+2,
\end{equation}where $L_1$ is the maximal (with respect to inclusion) dyadic interval in the postorder order interval $\B^N(\sigma_N(E_1),\sigma_N(E_2))$ that contains the left endpoint $\sigma_N(E_1)$. 
\end{theorem}

Lexicographic order intervals $\E(E_1,E_2)$ with large Carleson constant are given by endpoints $E_1,E_2$ which satisfy the property that $\Log2{\frac{1}{\abs{E_1}}}$ is much smaller than $\Log2{\frac{1}{\abs{E_2}}}$.
The upper bound in Theorem \ref{th:lexorder} depends for these order intervals only on the right endpoint $E_2$. Particularly, the upper bound is given by
\[\normBMON{T_{\sigma_N}\big\vert_{\M(\E)}}{2}\leq \Log2{\frac{1}{\abs{E_2}}}+2.\]  

Note that this upper bound can be obtained by applying Theorem \ref{th:lexorder} to the order interval $\E=\E(I_{0,0}, E_2)$.
\medskip

The duality relation of $H^1_N$ and $\BMON$, in particular the norm equivalence in equation (\ref{eq:duality}), and the interpolation resp.~extrapolation procedure in Theorem \ref{th:interpolation0} and Corollary \ref{co:interpolation1} give equivalent norm estimates as in Theorem \ref{th:operatornorm1} for the rearrangement operators on $H^p_N$, $0<p<\infty$, given by
 	\[T_{\tau_N,p}\colon\frac{h_I}{\abs{I}^{\frac{1}{p}}} \mapsto \frac{h_{\tau_N(I)}}{\abs{\tau_N(I)}^{\frac{1}{p}}} \quad \text{resp.}\quad T_{\sigma_N,p}\colon\frac{h_I}{\abs{I}^{\frac{1}{p}}} \mapsto \frac{h_{\sigma_N(I)}}{\abs{\sigma_N(I)}^{\frac{1}{p}}}.
 \]
 \begin{corollary}
 \label{co:Hp1}
 For all $0< p<\infty$ there exists a constant $C_p$ such that for all $N \in \NN_0$, $0 \leq \ell \leq N$ and $T=T_{\tau_N,p}\big\vert_{\M(\T^N_{\ell,0})}$ or $T=T_{\sigma_N,p}\big\vert_{\M(\D_{N-\ell})}$ the following holds
 \begin{equation}
 \frac{2^{-\abs{\frac{1}{p}-\frac{1}{2}}}}{C_p} \left(N-\ell+1\right)^{\abs{\frac{1}{p}-\frac{1}{2}}}\leq \norm{T}_{H^p_N}\leq C_p\left(N-\ell+1\right)^{\abs{\frac{1}{p}-\frac{1}{2}}}.
 \end{equation}
 \end{corollary}
 \begin{remark}
 By the convexification method (\cite{MR540367,MR871851}, see also \cite{MulPent} for the concrete specialisation to Hardy spaces) one obtains the same result as in Corollary \ref{co:Hp1} for the more general Triebel-Lizorkin spaces. 
 \end{remark}
 Corollary \ref{co:Hp1} gives, considering the general upper bound in Remark \ref{re:universalbound}, the same extremality statement for the rearrangement operators $T=T_{\tau_N,p}$ resp.~$T=T_{\sigma_N,p}$ on the spaces $H^p_N$, $0<p<\infty$. For all $0<p<\infty$ there exists a constant $B_p$ such that
 \begin{equation*}
\frac{2^{-\abs{\frac{1}{p}-\frac{1}{2}}}}{B_p}\,R^N(H^p_N)\leq  \norm{T}_{H^p_N} \leq R^N(H^p_N),
 \end{equation*}where $R^N(H^p_N)=\sup\Big\{\norm{T_{\tau}\colon H^p_N \rightarrow H^p_N}:\,\tau\colon \D_N \rightarrow \D_N\,\, \text{bijective}\Big\}.$

\section{Proof of Theorem \ref{th:operatornorm1}}
\subsection{Parameters associated with the postorder rearrangement}
\label{sec:postrear}
For the proof of the main theorem we need formulae that describe the map $\tau_N$ precisely. Recall that $\tau_N$ maps the $n^{th}$ dyadic interval in postorder onto the $n^{th}$ dyadic interval in lexicographic order. 
First of all we give formulae that describe the assignment of postorder ordinal numbers and lexicographic ordinal numbers to the dyadic intervals $I_{\ell,k} \in \D_N$.  
We denote by $a^{\ell}(k)$ the postorder ordinal number and by $b^{\ell}(k)$ the lexicographic ordinal number of the dyadic interval $I_{\ell,k}$. 

The assignment rule for a lexicographic ordinal number to a dyadic interval $I_{\ell,k}$ is given by
	\[b^{\ell}(k)=\left(\sum_{i=0}^{{\ell}-1}{2^i}\right)+k+1=2^{\ell}+k.
\]
We can determine from the ordinal number $b^{\ell}(k)$ the level ${\ell}$ and the position $k$ of the associated interval $I_{{\ell},k}$:
\begin{equation}
\label{eq:lexnew}
{\ell}=\lfloor \Log2{b^{\ell}(k)}\rfloor \quad \text{and} \quad 
k=b^{\ell}(k)-2^{\ell}.
\end{equation}

 The assignment rule for postorder ordinal numbers to the dyadic intervals is more difficult than in the lexicographic case. Let $j\in \NN$ with dyadic expansion $j=\sum \epsilon_i\, 2^i$. We define $m(j)=\min{\{i\in \NN:\,\epsilon_i \neq 0\}}$.  
\begin{lemma}
\label{le:alk}
Let $N \in \NN_0$, $0 \leq \ell \leq N$ and $0 \leq k \leq 2^{\ell}-1$. The postorder ordinal number of the dyadic interval $I_{\ell,k} \in \D_N$ is given by
 \begin{equation}
a^{\ell}(k)=(k+1)\,(2^{N-\ell+1}-1)+\sum_{j=1}^{k}{m(j)}.
\label{eq:al2}
\end{equation}
\end{lemma}
\begin{proof}
Let $1\leq j \leq 2^{\ell}-1$ and let 
\begin{equation}
\label{eq:tlj}
t^{\ell}(j)=a^{\ell}(j)-a^{\ell}(j-1)-1,
\end{equation}
where $a^{\ell}(j-1)$ and $a^{\ell}(j)$ are the postorder ordinal numbers of two successive dyadic intervals in level $\ell$. 
This gives the recursive formula $a^{\ell}(j)=a^{\ell}(j-1)+t^{\ell}(j)+1$
and thereby  
the assignment rule for the postorder ordinal number: 
\begin{equation}
\begin{split}
a^{\ell}(k)&=a^{\ell}(0)+k+\sum_{j=1}^{k}{t^{\ell}(j)}.
\end{split}
\label{eq:al}
\end{equation}
The definition of the postorder and the dyadic tree structure of $\D_N$ yield 
\begin{equation}
\label{eq:al0}
a^{\ell}(0)=2^{N-\ell+1}-1, \quad \text{for all } 0 \leq \ell \leq N.
\end{equation}

In the following we determine a formula for $t^{\ell}(j)$, $1\leq j \leq 2^{\ell}-1$. To this end, we give formulae that associate the postorder ordinal number of a dyadic interval with the postorder ordinal number of its parent.
We consider the dyadic interval $I_{\ell,k} \in \D_N$ with the postorder ordinal number $a^{\ell}(k)$ and its children $I_{\ell+1,2k}$ and $I_{\ell+1,2k+1}$ with the  postorder ordinal numbers $a^{\ell+1}(2k)$ and $a^{\ell+1}(2k+1)$.
By the definition of the postorder we have $a^{\ell+1}(2k)< a^{\ell+1}(2k+1)<  a^{\ell}(k)$. Furthermore, $a^{\ell+1}(2k)$ is smaller and $a^{\ell+1}(2k+1)$ is greater than the ordinal numbers of the descendants of $I_{\ell+1,2k+1}$. The number of descendants of $I_{\ell+1,2k+1}$ is $2^{N-\ell}-2$.  Hence, the definition of the postorder yields the following recursions:
\begin{align}
\label{eq:recursive}
a^{\ell}(k)&=a^{\ell+1}(2k+1)+1 \quad\text{ and } \quad a^{\ell}(k)=a^{\ell+1}(2k)+2^{N-\ell},
\end{align}where $0 \leq \ell \leq N$ and $0 \leq k \leq 2^{\ell}-1$.
Induction shows that for $1\leq i<\ell$
\begin{align}
\label{eq:ind1}
a^{\ell-i}\left(s-1\right)=a^{\ell}(2^i s-1)+i  \quad \text{and} \quad  a^{\ell-i}\left(s\right)=a^{\ell}(2^is)+2^{N-\ell+1}(2^i-1),
\end{align}where $1\leq s\leq 2^{\ell-i}-1$. 

Now we can determine an explicit formula for $t^{\ell}(j)$. 
If $j$ is odd,
then the formulae in (\ref{eq:recursive}) yield $a^{\ell-1}\left(\frac{j-1}{2}\right)=a^{\ell}(j)+1$ and $a^{\ell-1}\left(\frac{j-1}{2}\right)=a^{\ell}(j-1)+2^{N-\ell+1}$.
Therefore, by equation (\ref{eq:tlj})  
\begin{equation*}
t^{\ell}(j)=2^{N-\ell+1}-2, \quad \text{ if $j$ is odd.} 
\end{equation*}
If $j$ is even, then there exists an integer $i$, $1\leq i<\ell$, given by $i=m(j)$, and an odd integer $s$, $1\leq s\leq 2^{\ell-i}-1$ such that $j=2^is$. Equation (\ref{eq:tlj}) and the formulae in (\ref{eq:ind1}) yield
\begin{equation}
\label{eq:tl2i}
\begin{split}
t^{\ell}(s2^i)&=a^{\ell}(2^is)- a^{\ell}(2^is-1)-1\\
&=a^{\ell-i}(s)-2^{N-\ell+1}(2^i-1) - a^{\ell-i}(s-1)+i-1.
\end{split}
\end{equation}
Note that $s$ is odd. The formulae in (\ref{eq:recursive}) yield $a^{\ell-i}(s)=a^{\ell-i-1}\left(\frac{s-1}{2}\right)-1$ and $a^{\ell-i}(s-1)=a^{\ell-i-1}\left(\frac{s-1}{2}\right)-2^{N-\ell+i+1}$.
Therefore, by equation (\ref{eq:tl2i}) we have 
\begin{equation*}
t^{\ell}(s2^i)= 2^{N-\ell+1}-2+i.
\end{equation*}
Summarizing the above we have for all $1\leq j \leq 2^{\ell}-1$
\begin{equation}
\label{eq:tl}
	t^{\ell}(j)=m(j)+2^{N-{\ell}+1}-2.
\end{equation}Note that $m(j)=0$, if $j$ is odd.
Putting this into equation (\ref{eq:al}) yields the statement.

\end{proof}

Given the ordinal numbers of a dyadic interval with respect to both the postorder and the lexicographic order on $\D_N$ we can describe the postorder rearrangement $\tau_N$ as follows. 
Let $I_{\ell,k} \in \D_N$ and $a^{\ell}(k)$ the corresponding postorder ordinal number. 
Let $L$ and $K$ be non-negative integers such that $a^{\ell}(k)=2^L+K$. Recall that $2^L+K$ is the lexicographic ordinal number of the dyadic interval $I_{L,K} \in \D_N$. Then the \textit{postorder rearrangement $\tau_N$} is the bijective map on $\D_N$ that maps the dyadic interval $I_{{\ell},k}$ onto the dyadic interval $I_{L,K}$,  cf.~figure \ref{fig:fig1}.

In the following section we describe the determination of $L$ and $K$ such that $a^{\ell}(k)=2^L+K$. In the following we use the notation
\begin{align*}
\Level(a^{\ell}(k))= L \quad \text{and} \quad \Pos( a^{\ell}(k))=K.
\end{align*}  
According to (\ref{eq:lexnew}) we have
\begin{equation}
\Level(a^{\ell}(k))=\lfloor\Log2{(a^{\ell}(k))}\rfloor \quad \text{and}\quad  
\Pos(a^{\ell}(k))=a^{\ell}(k)-2^{\Level(a^{\ell}(k))}.
\label{eq:mallatnew}
\end{equation}
The following two Lemmata give formulae for $\Level(a^{\ell}(k))$ and $\Pos(a^{\ell}(k))$, which do not involve the postorder ordinal number $a^{\ell}(k)$ but only the level $\ell$ and the position $k$ of the corresponding dyadic interval $I_{\ell,k}$. 

\begin{lemma}
\label{le:level}
Let $N \in \NN_0$ and $0\leq {\ell} \leq N$.
For all $0\leq k \leq 2^{\ell}-1$ we have
 \begin{equation}
\label{eq:level}
\Level(a^{\ell}(k))=\left\lceil \Log2{(k+1)}\right\rceil+N-{\ell}.
 \end{equation}
\end{lemma}

\begin{proof}
The definition of the postorder yields $a^{\ell}(0)=2^{N-{\ell}+1}-1$. Therefore, by equation (\ref{eq:mallatnew}) we have $\Level(a^{\ell}(0))=N-{\ell}$.

Now we show that for all $1\leq s \leq \ell$ and for all $2^{s-1}\leq k \leq 2^s-1$
\begin{equation}
\label{eq:level2}
\Level{(a^{\ell}(k))}=s+N-{\ell}.
\end{equation}
By Lemma \ref{le:alk} we have
\begin{equation}
\label{eq:proof1}
\begin{split}
a^{\ell}(2^{s-1})
&=2^{N-{\ell}+s}+2^{N-{\ell}+1}-2^{s-1}-1+\sum_{j=1}^{2^{s-1}}m(j). 
\end{split}
\end{equation} Recall that for $j\in \NN$ given by its dyadic expansion $j=\sum \epsilon_i\,2^i$ we have $m(j)=\min{\{i\in \NN:\,\epsilon_i \neq 0\}}$. Hence, $m(j)=0$ for all odd integers $j$. We can split the sum on the right-hand side of equation (\ref{eq:proof1}) as follows
\begin{align*}
\sum_{j=1}^{2^{s-1}}m(j)=\sum_{j=1}^{2^{s-2}}{m(2 j)} =&\sum_{j_0=1}^{s-1}{m(2^{j_0})}+\sum_{j_1=1}^{s-2}\sum_{j_2=1}^{j_1-1}{m(2^{j_1}+2^{j_2}})+\cdots\\
&\cdots+\sum_{j_1=1}^{s-2}\cdots\sum_{j_{s-2}=1}^{j_{s-3}-1}{m(2^{j_1}+\cdots+2^{j_{s-2}})}.
\end{align*}
By definition, $m(2^{j_0})=j_0$ and $m(2^{j_1}+\cdots+2^{j_{i}})=j_{i}$ for $2 \leq i \leq s-2$.
Hence,
\begin{equation}
\label{eq:sumtnj}
\begin{split}
\sum_{j=1}^{2^{s-1}}{m(j)}&=\sum_{j_0=1}^{s-1}{j_0}+\sum_{j_1=1}^{s-2}\sum_{j_2=1}^{j_1-1}{j_2}+\cdots+\sum_{j_1=1}^{s-2}\cdots\sum_{j_{s-2}=1}^{j_{s-3}-1}{j_{s-2}}\\
&=\sum_{k=1}^{s-1}{\binom{s-1}{k}}=2^{s-1}-1.
\end{split}
\end{equation}
Putting this into formula (\ref{eq:proof1}) we get $a^{\ell}(2^{s-1})=2^{N-{\ell}+1}+2^{N-{\ell}+s}-2.$
Lemma \ref{le:alk} yields
\begin{equation*}
a^{\ell}(2^{s}-1)=2^{N-{\ell}+s+1}-2^{s}-m(2^s)+\sum_{j=1}^{2^{s}}{m(j)}.
\end{equation*}
Note that $m(2^s)=s$. By equation (\ref{eq:sumtnj}) we have $a^{\ell}(2^{s}-1)=2^{N-{\ell}+s+1}-s-1.$
Equation (\ref{eq:mallatnew}) yields
\begin{align*}
\Level(a^{\ell}(2^{s-1}))&=\lfloor\Log2{(2^{N-{\ell}+1}+2^{N-{\ell}+s}-2)}\rfloor=N-{\ell}+s,\\
\Level(a^{\ell}(2^{s}-1))&=\lfloor\Log2{(2^{N-{\ell}+1+s}-s-1)}\rfloor=N-{\ell}+s.
\end{align*}
Note that the map $k \mapsto \Level(a^{\ell}(k))$ is monotonically increasing for all $0 \leq k \leq 2^{\ell}-1$.
Therefore, (\ref{eq:level2}) is proven. 

For $0\leq k \leq 2^{\ell}-1$ there exists an integer $s$, $0\leq s \leq {\ell}$ such that $2^{s-1}\leq k \leq 2^s-1$.
Therefore, we get $s =\lceil \Log2{(k+1)}\rceil$ and $\Level(a^{\ell}(k))=\lceil \Log2{(k+1)}\rceil+N-{\ell}$.
\end{proof}
The next Lemma describes the determination of $K=\Pos{(a^{\ell}(k))}$. As stated previously, $\Pos{(a^{\ell}(k))}$ depends on $L=\Level{(a^{\ell}(k))}$, which was determined in Lemma \ref{le:level}.
Recall that for $j\in \NN$ with dyadic expansion $j=\sum \epsilon_i\, 2^i$ we have $m(j)=\min{\{i\in \NN:\,\epsilon_i \neq 0\}}.$
\begin{lemma}
\label{le:pos}
Let $N \in \NN_0$ and $0 \leq {\ell} \leq N$. Then $\Pos(a^{\ell}(0))=2^{N-\ell}-1$ and for all $0< k \leq 2^{\ell}-1$
\begin{align*}
\Pos(a^{\ell}(k))&=(k+1)\,(2^{N-\ell+1}-1)+2^{L-N+{\ell}-1}-2^L-1+\sum_{j=2^{L-N+{\ell}-1}+1}^{k}{m(j)}.
\end{align*}
\end{lemma}
\begin{proof}
Recall that $a^{\ell}(0)=2^{N-{\ell}+1}-1$ and $\Level(a^{\ell}(0))=N-\ell$. Therefore, by equation (\ref{eq:mallatnew}) we have $\Pos(a^{\ell}(0))=2^{N-\ell}-1$.

Fix $0<k\leq 2^{\ell}-1$ and let $L=\Level{(a^{\ell}(k))}$. Lemma \ref{le:level} yields $\Level{(a^{\ell}(j))}=L$, for all $2^{L-N+{\ell}-1}\leq j \leq 2^{L-N+{\ell}}-1$. 
Recall from the proof of Lemma \ref{le:alk} that $a^{\ell}(j)=a^{\ell}(j-1)+1+t^{\ell}(j)$,
where $t^{\ell}(j)=m(j)+2^{N-{\ell}+1}-2$. 
Hence, by equation (\ref{eq:mallatnew}) we have the following recursive formula
\begin{align*}
\Pos{(a^{\ell}(j))}&=\Pos{(a^{\ell}(j-1))}+1+t^{\ell}(j),\quad 2^{L-N+{\ell}-1}<j \leq 2^{L-N+{\ell}}-1
\end{align*}and therefore,  
\begin{equation}
\label{eq:posalk1}
\Pos{(a^{\ell}(k))}=\Pos{(a^{\ell}(2^{L-N+{\ell}-1}))}+k-2^{L-N+\ell-1}+\sum_{j=2^{L-N+\ell-1}+1}^kt^{\ell}(j).
\end{equation}
Since $t^{\ell}(j)=m(j)+2^{N-{\ell}+1}-2$, it follows that
\begin{equation}
\label{eq:posalk2}
\sum_{j=2^{L-N+\ell-1}+1}^kt^{\ell}(j)=k\,(2^{N-\ell+1}-2)+2^{L-N+\ell}-2^L+\sum_{j=2^{L-N+\ell-1}+1}^k m(j).
\end{equation}
Lemma \ref{le:alk} and equation (\ref{eq:sumtnj}) yield $a^{\ell}(2^{L-N+{\ell}-1})=2^L+2^{N-\ell+1}-2$ and by equation (\ref{eq:mallatnew}) we have
\begin{equation}
\label{eq:posalk3}
\Pos{(a^{\ell}(2^{L-N+{\ell}-1}))}=2^{N-{\ell}+1}-2.
\end{equation}
Putting equation (\ref{eq:posalk2}) and (\ref{eq:posalk3}) into equation (\ref{eq:posalk1}) yields the statement. 
\end{proof}

\subsection{Dyadic subtrees and their lowermost level in $\D_N$}
\label{sec:lowermostlevel}
\label{sec:subtrees}
In this section we examine the behaviour of the postorder rearrangement $\tau_N$ on complete dyadic subtrees in $\D_N$ given by 
\begin{equation}
 \label{eq:subdef}
 \T_{\ell,k}^N=\{I \in \D_N: I \subseteq I_{\ell,k}\}
 \end{equation}
and on their lowermost level in $\D_N$ given by
\begin{equation}
\label{eq:lexdef}
\E^N_{\ell,k}=\{I \in \D_N:\, I \subseteq I_{\ell,k},\, \abs{I}=2^{-N}\}.
\end{equation}
Note that $\E^N_{\ell,k}$ is a collection of disjoint dyadic intervals and hence, $\car{ \E^N_{\ell,k}}=1$. We know from (\ref{eq:subtree}) that $\car{ \T_{\ell,k}^N}=N-\ell+1$. 
We measure the behaviour of the rearrangement by the Carleson constants $\car{ \tau_N(\T_{\ell,k}^N)}$ and $\car{ \tau_N(\E^N_{\ell,k})}$. 
The following two theorems and the corresponding proofs reveal a remarkable phenomenon of the postorder rearrangement $\tau_N$. A complete dyadic subtree as well as its lowermost level in $\D_N$ is mapped under $\tau_N$ onto collections of dyadic intervals with large Carleson constant, if it contains the leftmost interval $I_{N,0}$, cf.~Theorem \ref{th:leftmost}. Otherwise, it is mapped under $\tau_N$ onto a collection of disjoint dyadic intervals of equal length, cf.~Theorem \ref{th:kg0}.

\begin{theorem}
\label{th:leftmost}
Let $N \in \NN_0$ and $0\leq \ell \leq N$. Then
\begin{equation*}
\car{ \tau_N(\T_{\ell,0}^N)}=N-\ell+1  \quad \text{and} \quad \car{ \tau_N(\E^N_{\ell,0})}\geq \frac{N-\ell+1}{2}.
\end{equation*}
\end{theorem}
\begin{proof}
Recall that $\tau_N$ maps the $n^{th}$ interval in postorder onto the $n^{th}$ interval in lexicographic order. The definition of the postorder yields that the dyadic intervals in $\T_{\ell,0}^N$ have the corresponding postorder ordinal numbers $1,\dots,2^{N-\ell+1}-1$, cf.~Section \ref{sec:postorder}. This are exactly the lexicographic ordinal numbers of the dyadic intervals in $\D_{N-\ell}.$ 
Hence, $\tau_N(\T^N_{\ell,0})=\D_{N-\ell}$ and $\car{ \tau_N(\T^N_{\ell,0})}=N-\ell+1$.

\smallskip
The lowermost level of $\T^N_{\ell,0}$ in $\D_N$ is given by 
\[\E^N_{\ell,0}=\{I_{N,r}:\,0\leq r \leq  2^{N-\ell}-1\}.\] 
By the description of the postorder rearrangement $\tau_N$ in Section~\ref{sec:postrear} we have  
\begin{equation*}
\tau_N(\E^N_{\ell,0})=\bigg\{I_{L,K}:\,L=\Level(a^N(r)),\,K=\Pos(a^N(r)),\,0\leq r \leq 2^{N-\ell}-1\bigg\}.
\end{equation*}
We know from Lemma \ref{le:level} and Lemma \ref{le:pos} that $\Level{(a^N(0))}=0$ and $\Pos{(a^N(0))}=0$. Therefore, $I_{0,0} \in \tau_N(\E^N_{\ell,0})$ and by definition  (\ref{eq:carleson}) 
\begin{equation}
\car{\tau_N(\E^N_{\ell,0})}\geq \frac{1}{\abs{I_{0,0}}} \sum_{\substack{J \subseteq I_{0,0},\\J \in \tau_N(\E^N_{\ell,0})}}{\abs{J}} = \sum_{J \in \tau_N(\E^N_{\ell,0})}{\abs{J}}.
\label{eq:tauNEN}
\end{equation}
Note that $\tau_N(\E^N_{\ell,0}) \subseteq \D_{N-\ell}$. We split the sum on the right hand side into levels and get 
\begin{equation}
\label{eq:splitlevel}
\begin{split}
\car{\tau_N(\E^N_{\ell,0})}\geq \sum_{m=0}^{N-\ell} 2^{-m}\abs{\B(m)},
\end{split}
\end{equation}where $\B(m)$ is the set of dyadic intervals in the collection $\tau_N(\E^N_{\ell,0})$ that have length $2^{-m}$.
We denote by $\A(m)$ the set of postorder ordinal numbers corresponding to $\B(m)$. Then $\abs{\B(m)}=\abs{\A(m)}$ and 
\begin{align*}
\A(m)&=\{a^N(r):\,0 \leq r \leq 2^{N-\ell}-1,\,\Level(a^N(r))=m\}.
\end{align*}
Obviously, $\abs{\A(0)}=1$. By Lemma \ref{le:level} we have $\abs{\A(m)}= 2^{m-1}$ for all $1\leq m\leq N-\ell$. Hence,
\begin{equation*}
\label{eq:splitlevelnew}
\car{\tau_N(\E^N_{\ell,0})}\geq 1+\sum_{m=1}^{N-\ell} 2^{-1}=\frac{N-\ell}{2}+1\geq \frac{N-\ell+1}{2}.
\end{equation*}
\end{proof}

\begin{remark}
Let $N \in \NN_0$. An easy computation shows that for $N-1 \leq \ell \leq N$ 
\[\car{ \tau_N(\E^N_{\ell,0})}=1+\frac{N-\ell}{2}.\]
Obviously, for $0\leq \ell \leq N-2$ we have the upper bound 
\[\car{ \tau_N(\E^N_{\ell,0})} \leq  N-\ell+1.\]
\end{remark}

\begin{conjecture}
Let $N \in \NN$, $N \geq 2$ and $0 \leq \ell \leq N-2$. 
The supremum in definition (\ref{eq:carleson}) for the Carleson constant $\car{\tau_N(\E^N_{\ell,0})}$ is attained for the interval $I_{1,0}$. This gives the following 
formula
\[\car{\tau_N(\E^N_{\ell,0})}=\frac{N-\ell}{2}+\frac{3}{2}-2^{-N+\ell+1}.\] 
\end{conjecture}

Now we consider those dyadic trees in $\D_N$ that are mapped under the postorder rearrangement $\tau_N$ onto collections of disjoint dyadic intervals.  
\begin{theorem}
\label{th:kg0}
Let $N \in \NN$, $0 < \ell \leq N$ and $0<k\leq 2^{\ell-1}$. Then 
\begin{equation}
\car{ \tau_N(\T_{\ell,k}^N)}= \car{\tau_N(\E^N_{\ell,k})}=1.
\end{equation}
\end{theorem}
\begin{proof}
The idea of the proof is that for $k>0$ there exists $s \in \NN_0$ such that the collection  $\T_{\ell,k}^N$ is a subset of the collection of the dyadic intervals of length $s+N-\ell$.

The complete dyadic subtree $\T_{\ell,k}^N$ is given by
\[\T_{\ell,k}^N=\{I_{m,r}:\ell\leq m \leq N,\,k\,2^{m-\ell}\leq r \leq (k+1)\,2^{m-\ell}-1\}.\]
We associate the collection $\T_{\ell,k}^N$ with the set of postorder ordinal numbers 
\begin{equation}
\label{eq:amr}
\{a^{m}(r):\ell \leq m\leq N,\,  k\,2^{m-\ell}\leq r \leq (k+1)\,2^{m-\ell}-1\}.
\end{equation}
By the description of the postorder rearrangement $\tau_N$ in Section~\ref{sec:postrear}, we have  
\[\tau_N(\T_{\ell,k}^N)=\{I_{L,K}:\,L=\Level{(a^m(r))},\,K=\Pos{(a^m(r))}\}.\]
Let $s=\left\lceil\Log2{(k+1)}\right\rceil$. We show that for all $m$ and $r$ as in (\ref{eq:amr}) we have
\[\Level{(a^m(r))}=s+N-\ell.\]
Let $\overline{s}=m-\ell+\left\lceil\Log2{(k+1)}\right\rceil$ so that $2^{\overline{s}-1} \leq k\,2^{m-\ell}$ and $(k+1)2^{m-\ell}-1\leq 2^{\overline{s}}-1$.
By Lemma \ref{le:level} it follows that for all $k\,2^{m-\ell}\leq r \leq (k+1)\,2^{m-\ell}-1$ we have
\begin{equation*}
\Level(a^m(r))=\overline{s}+N-m=s+N-\ell.
\end{equation*}
The image of $\T_{\ell,k}^N$ is then given by
\[\tau_N(\T_{\ell,k}^N)=\{I_{s+N-\ell,K}:K=\Pos{(a^m(r))} \}. \]
$\tau_N(\T_{\ell,k}^N)$ is a collection of disjoint dyadic intervals and therefore, $\car{ \tau_N(\T_{\ell,k}^N)}=1$.
Since $\E^N_{\ell,k} \subseteq \T^N_{\ell,k}$, it follows that $\tau_N(\E^N_{\ell,k}) \subseteq \tau_N(\T_{\ell,k}^N)$. Therefore, $\tau_N(\E^N_{\ell,k})$ is also a collection of disjoint dyadic intervals with $\car{\tau_N(\E^N_{\ell,k}) }=1$.
\end{proof}

\subsection{The proof of Theorem \ref{th:operatornorm1}}
Finally, we have all ingredients that we need to prove the statement of Theorem \ref{th:operatornorm1}. 
For convenience we recall the statement. The rearrangement operators $T=T_{\tau_N}\big\vert_{\M(\T_{\ell,0}^N)}$ and $T=T_{\sigma_N}\big\vert_{\M(\D_{N-\ell})}$ satisfy the following norm estimates
\begin{equation*}
\frac{1}{\sqrt{2}}(N-\ell+1)^{\frac{1}{2}}\leq\normBMON{T}{}\leq (N-\ell+1)^{\frac{1}{2}}.
\end{equation*}
Recall that $\M(\T_{\ell,0}^N)=\lin{\{h_I: I \in \T_{\ell,0}^N\}}$ and $\M(\D_{N-\ell})=\lin{\{h_I: I \in \D_{N-\ell}\}}$.
The proof uses the norm estimates for rearrangement operators on $\BMON$ given in Section \ref{sec:rearop} and the estimates for Carleson constants given in Theorem \ref{th:leftmost} and Theorem \ref{th:kg0}.  

\begin{proof}[Proof of Theorem \ref{th:operatornorm1}]
 Let $x \in \M(\T_{\ell,0}^N)$. 
The norm estimate (\ref{eq:rear1}) and the statement of Theorem \ref{th:leftmost} yield
\[\normBMO{T_{\tau_N}x}{2}\leq \car{ \tau_N(\T_{\ell,0}^N)} \normBMO{x}{2}\leq (N-\ell+1)\normBMO{x}{2}.\]
This gives the upper bound
\[\normBMON{T_{\tau_N}\big\vert_{\M(\T_{\ell,0}^N)}}{}\leq (N-\ell+1)^{\frac{1}{2}}.\]
Equation (\ref{eq:rear}) gives the lower bound
\[\normBMON{T_{\tau_N}\big\vert_{\M(\T_{\ell,0}^N)}}{2}\geq \sup_{\substack{\C \subseteq \T_{\ell,0}^N,\\ \text{non-empty}}}\frac{\car{\tau_N(\C)}^{\frac{1}{2}}}{\car{\C}^{\frac{1}{2}}}.\]
We consider the lowermost level $\E^N_{\ell,0}$ of the complete dyadic subtree $\T^N_{\ell,0}$ in $\D_N$. 
 Obviously, $\E^N_{\ell,0}\subseteq \T^N_{\ell,0}$. We know that $\car{\E^N_{\ell,0}}=1$. Hence, by the statement of Theorem~\ref{th:leftmost} we have
\[\normBMON{T_{\tau_N}\big\vert_{\M(\T_{\ell,0}^N)}}{2}\geq \car{\tau_N(\E^N_{\ell,0})}\geq \frac{1}{2}(N-\ell+1).\]

Let $x\in \M(\D_{N-\ell})$. 
The norm estimate (\ref{eq:rear1}) yields
 \[\normBMO{T_{\sigma_N}x}{2}\leq \car{ \sigma_N(\D_{N-\ell})} \normBMO{x}{2}.\] 
By the proof of Theorem~\ref{th:leftmost} it follows that $\tau_N(\T_{\ell,0}^N)=\D_{N-\ell}$. Since $\sigma_N=\tau_N^{-1}$, we have $\sigma_N(\D_{N-\ell})= \T_{\ell,0}^N$ and $\car{ \sigma_N(\D_{N-\ell})}=\car{ \T_{\ell,0}^N} = N-\ell+1$.
Hence,
\[\normBMON{T_{\sigma_N}\big\vert_{\M(\D_{N-\ell})}}{}\leq (N-\ell+1)^{\frac{1}{2}}.\]
Equation (\ref{eq:rear}) gives the lower bound
\[\normBMON{T_{\sigma_N}\big\vert_{\M(\D_{N-\ell})}}{2}\geq \sup_{\substack{\C \subseteq \D_{N-\ell},\\ \text{non-empty}}}\frac{\car{\sigma_N(\C)}^{\frac{1}{2}}}{\car{\C}^{\frac{1}{2}}}.\]
Let $\ell<N$. The proof of Theorem~\ref{th:kg0} asserts that 
   $\tau_N(\T^N_{\ell+1,1}) \subseteq \D_{N-\ell}$ and $\car{ \tau_N(\T^N_{\ell+1,1}) }=1$.
Hence, we have the lower bound
\[\normBMON{T_{\tau_N}\big\vert_{\M(\D_{N-\ell})}}{2}\geq \car{ \T^N_{\ell+1,1}}=N-\ell\geq \frac{1}{2}(N-\ell+1).\]
The case $\ell=N$ is trivial. 
\end{proof}

\section{Proof of Theorem \ref{th:lexorder}}
As mentioned in Section \ref{sec:proof1}, the proof of Theorem \ref{th:lexorder} uses a geometric representation of order intervals with respect to the postorder, $\preceq$. This geometric representation is given in Proposition \ref{prop:maxint} and Definition \ref{def:cone} as follows.
For every postorder order interval
\[\B^N(I_1,I_2)=\{I \in \D_N: I_1 \preceq I \preceq I_2\},\]
there exists a collection of maximal intervals $\LL=\{L_1,\dots,L_m\}$ such that
\begin{equation*}
\B^N(I_1,I_2)=\C(I_1,L_1)\cup\R(I_1,L_1)\cup_{i=2}^{m}\M_i,
\end{equation*}where $\C(I_1,L_1)$ is the cone of dyadic intervals between $I_1$ and $L_1$, $\R(I_1,L_1)$ is the right fill-up of the cone and $\M_i$ is the complete dyadic subtree with root $L_i$ given by $\M_i=\{I \in \D_N:I \subseteq L_i\}$.

For the norm estimate in Theorem \ref{th:lexorder} we need an estimate for the Carleson constant $\car{ \B^N(I_1,I_2)}$. By the geometric representation of $B^N$ given above we have that 
the Carleson constant $\car{ \B^N(I_1,I_2)}$ is related to the Carleson constant of the cone and the right fill-up. Therefore, we start examining the Carleson constant $\car{ \C(I,J)\cup\R(I,J)}$ for two non-disjoint dyadic intervals $I,J \in \D_N$.
  \begin{theorem}
 \label{th:coneright}
 Let $N \in \NN_0$.
  Let $I,J \in \D_N$ and $I \subseteq J$. 
    If $\R(I,J) \neq \emptyset$, then
  \begin{equation}
  \label{eq:CR}
  N-\Log2{\frac{1}{\abs{I}}}+1\leq \car{ \C(I,J)\cup \R(I,J)} \leq N-\Log2{\frac{1}{\abs{J}}}+2.
  \end{equation}
 \end{theorem}
 \begin{proof}
  The definition of the Carleson constant (\ref{eq:carleson}) yields
   \begin{equation}
    \label{eq:carlesonCR}
    \car{ \R(I,J)}\leq \car{ \C(I,J)\cup\R(I,J)} \leq  \car{ \R(I,J)}+ \car{ \C(I,J)}.
    \end{equation}
 
  Recall that the cone $\C(I,J)$ is a collection of dyadic intervals $\C=\{C_1,\dots,C_n\}$, where $n=\Log2{\frac{\abs{J}}{\abs{I}}}+1$, which satisfies the following properties: $C_1=I$, $C_n=J$, $\abs{C_i}=\frac{1}{2}\abs{C_{i+1}}$ and $C_i \subset C_{i+1}$ for $1\leq i \leq n-1$. This yields 
   \begin{equation*}
   \label{eq:cone}
  \car{ \C(I,J)}=\sup_{i=1,\dots,n}{\frac{1}{\abs{C_i}}\sum_{J\subseteq C_i,\, J \in \C}{\abs{J}}}=\sup_{i=1,\dots,n}{\frac{1}{\abs{C_i}}\sum_{s=1}^i{\abs{C_s}}}.
  \end{equation*}
  Since $\abs{C_i}=2^{i-1}\abs{C_1}$, it follows that $\car{ \C(I,J)}\leq 2$. 

\noindent
  The right fill-up $\R(I,J)$ of the cone is the collection of dyadic intervals $\bigcup_{i=1}^{n-1}\U_{i+1}$, where $\U_{i+1}=\emptyset$, if $C_i$ is the right half of $C_{i+1}$ and $\U_{i+1}=\{U \in \D_N: U \subseteq C_{i+1}\setminus C_{i}\}$,
    if $C_i$ is the left half of $C_{i+1}$. Note that by definition $\U_{i}\cap\U_j=\emptyset$ for every $i\neq j$. Therefore,
      \begin{equation}
         \label{eq:carright}
         \car{ \R(I,J)}=\sup_{i=1,\dots,n-1}\car{ \U_{i+1}}.
       \end{equation}
     If $\,\U_{i+1}\neq \emptyset$, then $\U_{i+1}$ is a dyadic subtree of $\D_N$ with root $C_{i+1}\setminus C_i$ and depth $N-\Log2{\frac{1}{\abs{C_{i+1}\setminus C_i}}}$. We know that $\abs{C_{i+1}\setminus C_i}=\abs{C_i}$ and $\abs{C_i}=2^{i-1}\abs{I}$. Therefore, by equation (\ref{eq:subtree}) we have $ \car{ \U_{i+1}}
          =N-\Log2{\frac{2^{1-i}}{\abs{I}}}+1=N+i-\Log2{\frac{1}{\abs{I}}}$.
      Hence, by equation (\ref{eq:carright})
      \[N+1-\Log2{\frac{1}{\abs{I}}} \leq \car{ \R(I,J)}\leq N+n-1-\Log2{\frac{1}{\abs{I}}}. \]
      Recall that $n=\Log2{\frac{\abs{J}}{\abs{I}}}+1$. This gives the upper bound
      \[\car{ \R(I,J)}\leq  N-\Log2\frac{1}{\abs{J}}.\]
   Summarizing we have (\ref{eq:CR}).
  \end{proof}
  
The statement of Proposition \ref{prop:maxint}  and the estimates from Theorem \ref{th:coneright} yield the following estimates for the Carleson constant $\car{ \mathcal{B}^N(I_1,I_2)}$.
  \begin{theorem}
  \label{th:carlesonorderint}
  Let $N \in \NN_0$ and $\B^N(I_1,I_2)=\{I \in \D_N: I_1 \preceq I \preceq I_2\}$, where $I_1,I_2 \in \D_N$ with $I_1 \preceq I_2$. 
   Let $L_1$ be the maximal interval in the order interval $\B^N(I_1,I_2)$ such that $I_1 \subseteq L_1$. Then
  \begin{equation}
  \label{eq:carlesonorderint}
N-\Log2{\frac{1}{\abs{I_1}}}+1\leq \car{ \mathcal{B}^N(I_1,I_2)}\leq N-\Log2{\frac{1}{\abs{L_1}}}+2.
  \end{equation}
   \end{theorem} 
  \begin{proof}
  Let $\LL=\{L_1,\dots, L_m\}$ be the maximal (with respect to inclusion) elements of $\B^N(I_1,I_2)$, as given in Proposition \ref{prop:maxint} . Since
  \begin{equation*}
  \B^N(I_1,I_2)=\C(I_1,L_1)\cup\R(I_1,L_1)\cup_{i=2}^{m}\M_i,
  \end{equation*}where $\M_i=\{I \in \D_N:I \subseteq L_i\}$, and 
  since $\M_i\cap \M_j=\emptyset$ for all $i\neq j$ and $\M_i\cap(\C(I_1,L_1)\cup\R(I_1,L_1))=\emptyset$ for all $i$, we have
  \begin{equation}
  \car{ \mathcal{B}^N(I_1,I_2)}=\max \{\car{ \C(I_1,L_1)\cup\R(I_1,L_1)},\max_{i=2,\dots,m}\car{\M_i}\}.
  \end{equation}
  
  If $I_1 \subseteq I_2$, then there is only one maximal interval $L_1=I_2$. Hence, $m=1$ and $I \in \B^N(I_1,I_2)$ if and only if $I \in \C(I_1,I_2)\cup\R(I_1,I_2)$, cf. proof of Proposition \ref{prop:maxint} in \cite{MR1389531}. Therefore, we have
  \begin{equation}
  \label{eq:subset}
  \car{ \mathcal{B}^N(I_1,I_2)}=\car{ \C(I_1,I_2)\cup\R(I_1,I_2)}.
  \end{equation}
  Theorem \ref{th:coneright} yields the statement. 
  
  If $I_1\cap I_2 =\emptyset$, then there exist maximal intervals $\LL=\{L_1,\dots,L_m\}$, $m\geq 2$. For $2\leq i \leq m$, $\M_i$ is a dyadic subtree of $\D_N$ with root $L_i$ and depth $N-\Log2{\frac{1}{\abs{L_i}}}$. Equation (\ref{eq:subtree}) yields $\car{ \M_i}=N-\Log2{\frac{1}{\abs{L_i}}}+1$.
  Proposition \ref{prop:maxint} yields $\abs{L_m}\leq \abs{L_{m-1}}<\cdots<\abs{L_2}$. Hence, $\car{\M_2}=\max_{i=2,\dots,m}\car{\M_i}$
and
  \begin{equation*}
  \car{ \mathcal{B}^N(I_1,I_2)}=\max \{\car{ \C(I_1,L_1)\cup\R(I_1,L_1)},\car{\M_2}\}.
  \end{equation*}Note that $\car{\M_2}=N-\Log2{\frac{1}{\abs{L_2}}}+1$. 
   By Theorem \ref{th:coneright} we have the following lower and upper bound.
    \begin{equation}
    \label{eq:lhsorderintervall}
      \car{ \mathcal{B}^N(I_1,I_2)}\geq \max \{N-\Log2{\frac{1}{\abs{I_1}}}+1,\, N-\Log2{\frac{1}{\abs{L_2}}}+1\}
    \end{equation} and
     \begin{equation}
      \label{eq:rhsorderintervall}
        \car{ \mathcal{B}^N(I_1,I_2)}\leq \max \{N-\Log2{\frac{1}{\abs{L_1}}}+2,\, N-\Log2{\frac{1}{\abs{L_2}}}+1\}.
      \end{equation}
      Inequality (\ref{eq:lhsorderintervall}) gives the left-hand side of (\ref{eq:carlesonorderint}).
 Proposition \ref{prop:maxint} states that $\abs{L_2} \leq \abs{L_1}$. Therefore, (\ref{eq:rhsorderintervall}) yields the right-hand side of (\ref{eq:carlesonorderint}).
\end{proof}

\subsection{The proof of Theorem \ref{th:lexorder}}
Now we have all ingredients for the proof of Theorem \ref{th:lexorder}. For convenience we give the statement of the Theorem. We have the following operator norm estimate for the rearrangement operator $T_{\sigma_N}$ acting on lexicographic order intervals $\E(E_1,E_2)$ given by the endpoints $E_1, E_2 \in \D_N$ with $E_1 \leq_l E_2$:
\begin{equation*}
\normBMON{T_{\sigma_N}\big\vert_{\M(\E)}}{2}\leq N-\Log2{\frac{1}{\abs{L_1}}}+2,
\end{equation*}where $L_1$ is the maximal (with respect to inclusion) dyadic interval in the postorder order interval $B^N(\sigma_N(E_1),\sigma_N(E_2))$ that contains the left endpoint $\sigma_N(E_1)$. 
Recall that $\M(\E)=\lin{\{I \in \D_N: I \in \E\}}.$

\begin{proof}[Proof of Theorem \ref{th:lexorder}]
Let $x\in \M(\E).$ The estimates of rearrangement operators on $\BMON$ in Section \ref{sec:rearop} give the upper bound
\begin{equation}
\label{eq:help1}
\normBMO{T_{\sigma_N}x}{2}\leq \car{\sigma_N(\E)}\, \normBMO{x}{2}.
\end{equation}
$\sigma_N$ is the bijective map on $\D_N$ that maps lexicographic order intervals onto postorder order intervals. Hence, for every lexicographic order interval $\E=\E(E_1,E_2)$ there exists a unique postorder order interval $\B=\B^N(\sigma_N(E_1),\sigma_N(E_2))$ so that $\sigma_N(\E)=\B$. Hence, by equation (\ref{eq:help1}) and Theorem \ref{th:carlesonorderint} we have
\begin{equation}
\normBMO{T_{\sigma_N}x}{2}\leq \car{\B}\, \normBMO{x}{2} \leq ( N-\Log2{\frac{1}{\abs{L_1}}}+2)\, \normBMO{x}{2},
\end{equation}
where $L_1$ is the maximal interval in $\B^N(\sigma_N(E_1),\sigma_N(E_2))$ with $\sigma_N(E_1) \subseteq L_1$. 
\end{proof}

\section*{Outlook}
 In \cite{MR1466662} P.F.X.~M\"uller gives a flexible geometric condition on the rearrangement $\tau$ which describes the isomorphic properties of the rearrangement operator $T_{\tau}$, see also \cite{MR2559130}.
 It is an open problem to uncover an equally flexible geometric condition on $\tau$ which is connected to isometric properties of $T_{\tau}$. Our present work could be interpreted as a first step towards the solution of this problem. 
\medskip

\subsection*{Acknowledgements}
This paper is part of the author's PhD thesis
written at the Department of Analysis, J.~Kepler University Linz. This research
has been supported by the Austrian Science foundation (FWF) Pr.Nr.~P22549. 

I thank my PhD advisor P.F.X.~M\"uller for many helpful discussions and suggestions during the preparation of this paper. 

I am grateful to G.~Schechtman for information on the postorder and to M.~Passenbrunner for suggestions on notation and description concerning the postorder. 

\bibliographystyle{plain}
\bibliography{postorder}

\begin{thebibliography}{10}

\bibitem{berman}
K.~Berman and J.~Paul.
\newblock {\em Algorithms : sequential, parallel, and distributed}.
\newblock Thomson/Course Technology, Boston, Mass, 2005.

\bibitem{MR871851}
B.~Cuartero and M.A. Triana.
\newblock {$(p,q)$}-convexity in quasi-{B}anach lattices and applications.
\newblock {\em Studia Math.}, 84(2):113--124, 1986.

\bibitem{MR0447953}
C.~Fefferman and E.M. Stein.
\newblock {$H^{p}$} spaces of several variables.
\newblock {\em Acta Math.}, 129(3-4):137--193, 1972.

\bibitem{MR0448538}
A.M. Garsia.
\newblock {\em Martingale inequalities: {S}eminar notes on recent progress}.
\newblock W. A. Benjamin, Inc., Reading, Mass.-London-Amsterdam, 1973.
\newblock Mathematics Lecture Notes Series.

\bibitem{MR2559130}
S.~Geiss and P.F.X. M{\"u}ller.
\newblock Extrapolation of vector-valued rerrangement operators.
\newblock {\em J. Lond. Math. Soc. (2)}, 80(3):798--814, 2009.

\bibitem{MR2183484}
S.~Geiss, P.F.X. M{\"u}ller, and V.~Pillwein.
\newblock A remark on extrapolation of rearrangement operators on dyadic
  {$H^s$}, {$0<s\leq 1$}.
\newblock {\em Studia Math.}, 171(2):197--205, 2005.

\bibitem{MR3054318}
A.~Kamont and P.F.X. M{\"u}ller.
\newblock Rearrangements with supporting trees, isomorphisms and shift
  operators.
\newblock {\em Math. Z.}, 274(1-2):57--83, 2013.

\bibitem{MR2245382}
D.E. Knuth.
\newblock {\em The art of computer programming. {V}ol. 1. {F}asc. 1}.
\newblock Addison-Wesley, Upper Saddle River, NJ, 2005.
\newblock MMIX, a RISC computer for the new millennium.

\bibitem{MR540367}
J.~Lindenstrauss and L.~Tzafriri.
\newblock {\em Classical {B}anach spaces. {II}}, volume~97 of {\em Ergebnisse
  der Mathematik und ihrer Grenzgebiete [Results in Mathematics and Related
  Areas]}.
\newblock Springer-Verlag, Berlin, 1979.
\newblock Function spaces.

\bibitem{Mallat1}
S.G. Mallat.
\newblock A theory for multiresolution signal decomposition: the wavelet
  representation.
\newblock {\em IEEE Trans. Pattern Analysis and Machine Intelligence},
  2(7):674--693, 1989.

\bibitem{MR1219953}
Y.~Meyer.
\newblock {\em Wavelets}.
\newblock Society for Industrial and Applied Mathematics (SIAM), Philadelphia,
  PA, 1993.
\newblock Algorithms \& applications, Translated from the French and with a
  foreword by Robert D. Ryan.

\bibitem{MR955660}
P.F.X. M{\"u}ller.
\newblock On projections in {$H^1$} and {BMO}.
\newblock {\em Studia Math.}, 89(2):145--158, 1988.

\bibitem{MR1466662}
P.F.X. M{\"u}ller.
\newblock Rearrangements of the {H}aar system that preserve {BMO}.
\newblock {\em Proc. London Math. Soc. (3)}, 75(3):600--618, 1997.

\bibitem{MR2157745}
P.F.X. M{\"u}ller.
\newblock {\em Isomorphisms between {$H^1$} spaces}, volume~66 of {\em Instytut
  Matematyczny Polskiej Akademii Nauk. Monografie Matematyczne (New Series)
  [Mathematics Institute of the Polish Academy of Sciences. Mathematical
  Monographs (New Series)]}.
\newblock Birkh\"auser Verlag, Basel, 2005.

\bibitem{MR2927805}
P.F.X. M{\"u}ller.
\newblock Extrapolation of vector-valued rearrangement operators {II}.
\newblock {\em J. Lond. Math. Soc. (2)}, 85(3):722--736, 2012.

\bibitem{MulPent}
P.F.X. M{\"u}ller and J.~Penteker.
\newblock p-summing multiplication operators, dyadic {H}ardy spaces and atomic
  decomposition.
\newblock {\em Houston Journal Math.}, 41(2):639--668, 2015.

\bibitem{MR1389531}
P.F.X. M{\"u}ller and G.~Schechtman.
\newblock A remarkable rearrangement of the {H}aar system in {$L_p$}.
\newblock {\em Proc. Amer. Math. Soc.}, 125(8):2363--2371, 1997.

\bibitem{MR1576148}
R.E.A.C. Paley.
\newblock A {R}emarkable {S}eries of {O}rthogonal {F}unctions ({I}).
\newblock {\em Proc. London Math. Soc.}, S2-34(1):241, 1932.

\bibitem{MR1063121}
F.~Schipp.
\newblock On equivalence of rearrangements of the {H}aar system in dyadic
  {H}ardy and {BMO} spaces.
\newblock {\em Anal. Math.}, 16(2):135--141, 1990.

\bibitem{MR510261}
E.M. Sem{\"e}nov.
\newblock Equivalence in {$L^p$} of permutations of the {H}aar system.
\newblock {\em Dokl. Akad. Nauk SSSR}, 242(6):1258--1260, 1978.

\bibitem{MR648492}
E.M. Sem{\"e}nov and B.~St{\"o}kert.
\newblock Permutations of the {H}aar system in spaces {$L_{p}$}.
\newblock {\em Anal. Math.}, 7(4):277--295, 1981.

\bibitem{MR2400818}
J.S. Walker.
\newblock {\em A primer on wavelets and their scientific applications}.
\newblock Studies in Advanced Mathematics. Chapman \& Hall/CRC, Boca Raton, FL,
  second edition, 2008.

\end{thebibliography}

\end{document}